\documentclass[a4paper,11pt]{amsart}
\usepackage{amsmath,amsthm,amssymb,amsfonts,color,esint}
\usepackage[pdftex]{graphicx}
\usepackage{float}
\usepackage{enumitem}

\AtBeginDocument{%
   \def\MR#1{}
}

\usepackage[colorlinks=true, allcolors=blue]{hyperref}
\usepackage{amsrefs}

\usepackage[ansinew]{inputenc}
\usepackage[dvips]{epsfig}
\usepackage{graphicx}
\usepackage[english]{babel}
\usepackage{tikz}
\usepackage{hyperref}
\usepackage{url}
\newcommand{\abbr}[1]{{\sc\lowercase{#1}}}

\numberwithin{equation}{section}

\newtheorem{theorem}{Theorem}[section]

\newtheorem{conjecture}[theorem]{Conjecture}

\newtheorem{definition}[theorem]{Definition}

\newtheorem{lemma}[theorem]{Lemma}

\newtheorem{proposition}[theorem]{Proposition}
\newtheorem{remark}[theorem]{Remark}

\newcommand{\E}{\mathbb{E}}

\newcommand{\G}{\mathbb{G}}

\newcommand{\N}{\mathbb{N}}

	\renewcommand{\P}{\mathbb{P}}

\newcommand{\R}{\mathbb{R}}
\newcommand{\T}{\mathbb{T}}

\newcommand{\Z}{\mathbb{Z}}

\newcommand{\wh}{\widehat}
\newcommand{\wt}{\widetilde}
\newcommand{\ovl}{\overline}
\newcommand{\ep}{\epsilon}

\newcommand{\cE}{\mathcal{E}}
\newcommand{\cF}{\mathcal{F}}
\newcommand{\cG}{\mathcal{G}}
\newcommand{\cH}{\mathcal{H}}

\newcommand{\cN}{\mathcal{N}}

\newcommand{\bS}{\mathcal{S}}

\newcommand{\cU}{\mathbf{U}}

\newcommand{\bx}{\mathbf{x}}
\newcommand{\bX}{\mathbf{X}}

\newcommand{\bv}{\mathbf{v}}

\begin{document}

\title{The Kuramoto model on dynamic random graphs}


\author[P. Groisman]{Pablo Groisman}
\address{Departamento de Matem\'atica\hfill\break \indent Facultad de Ciencias Exactas y Naturales\hfill\break \indent Universidad de Buenos Aires\hfill\break \indent IMAS-UBA-CONICET\hfill\break \indent Buenos Aires, Argentina}
\email{pgroisma@dm.uba.ar}

\author[R. Huang]{Ruojun Huang}
\address{Faculty of Mathematics and Computer Science\hfill\break \indent University of M\"unster, Germany}
\email{ruojun.huang@uni-muenster.de}

\author[H. Vivas]{Hern\'an Vivas}
\address{Centro Marplatense de Investigaciones Matem\'aticas/Conicet}
\address{De\'an Funes 3350, 7600, Mar del Plata, Argentina}
\email{havivas@mdp.edu.ar}



\begin{abstract}

We propose a Kuramoto model of coupled oscillators on a time-varying graph, whose dynamics are dictated by a Markov process in the space of graphs. The simplest representative is considering a base graph and then the subgraph determined by $N$ independent random walks on the underlying graph. We prove a synchronization result for solutions starting from a phase-cohesive set independent of the speed of the random walkers, an averaging principle and a global synchronization result with high probability for sufficiently fast processes. We also consider Kuramoto oscillators in a dynamical version of the Random Conductance Model.

\end{abstract}

\maketitle

\section{Introduction and main results}

Synchronization of systems of coupled oscillators is a phenomenon that has attracted the mathematical and scientific communities for centuries. On one hand, such problems have an intrinsic mathematical interest, in part due to their complexity and the rich variety of mathematical questions that they entail. On the other hand, they are useful models for understanding a wide range of physical and biological phenomena \cites{mirollo1990synchronization, winfree1967biological, acebron2005kuramoto, bullo2020lectures, arenas2008synchronization, dorflerSurvey, strogatz2004sync}. 

One of the most popular models for describing synchronization of a system of $N$ coupled oscillators is the so called Kuramoto model, given by the ODE system
\begin{equation}
\label{kuramoto.complete}
\frac{d}{dt}\theta_i=\omega_i+\frac{K}{N}\sum_{j=1}^N\sin(\theta_j-\theta_i), \quad i=1,\ldots,N. 
\end{equation}
Here $\theta_i$ represents the phase of the $i-$th oscillator, $\omega_i$ its natural frequency and the coupling constant $K$ accounts for the strength of the coupling between any two oscillators. Note that in \eqref{kuramoto.complete} every oscillator is coupled with every other oscillator and with the same coupling constant.

With origins in the study of chemical reactions and the behavior of biological populations with oscillatory features, \cites{kuramoto1975self,kuramoto1984chemical}, the Kuramoto model proved to be applicable in the description of phenomena in areas as varied as neuroscience \cites{cumin2007generalising,breakspear2010generative}; superconductors theory \cite{wiesenfeld1996synchronization}; the beating rhythm of pacemaker cells in human hearts \cite{peskin1975mathematical} and the spontaneous flashing of populations of fireflies \cite{mirollo1990synchronization}. The reader is referred to  the surveys  \cite{dorfler2013synchronization,acebron2005kuramoto,rodrigues2016kuramoto} and the references therein for a more complete picture of the available advances on the topic. 



Equation \eqref{kuramoto.complete} represents a mean field model and, as such, it can not reproduce all the subtleties of specific networks with particular topologies that occur in reality. Then, it is natural to consider more general versions of \eqref{kuramoto.complete} where the underlying connectivity structure is given by different kinds of graphs. For an undirected graph $\mathbb G$ with set of vertices $V$ and edges $E$, the Kuramoto model \eqref{kuramoto.complete} on $\mathbb G$ is written
\begin{equation}
\label{kuramoto.graph}
\frac{d}{dt}\theta_u=\omega_u+\frac{K}{|V|}\sum_{v\colon \{u,v\}\in E}\sin(\theta_v-\theta_u), \quad u\in V
\end{equation}
where $|V|$ is the cardinality of $V$.

A natural question is if, for a given graph $\mathbb G$ and $\omega_i=\omega_j$ for every $1\le i,j\le N$,  the all-in-phase state is the unique stable equilibrium for \eqref{kuramoto.graph} or on the contrary, there are other stable equilibria.

One of the key issues to address in this setting is the relationship between the topology (and in particular the connectivity) of the graph that represents the network of oscillators and the achievement of phase synchronization.
This relation is not a simple one and, for instance, more connectivity does not necessarily imply a higher tendency towards synchronization nor the other way around \cites{wiley2006size,taylor2012there}.

Remarkably, in \cite{taylor2012there,ling2019landscape, kassabov2021sufficiently} it is proved that for highly connected graphs (each node is connected to a fraction of the nodes of size at least $\mu$) the all-in-phase state is the only stable equilibrium when $\mu$ is sufficiently large, implying that for almost all initial conditions the system converges to this equilibrium. In \cite{kassabov2021sufficiently} it is shown that $\mu > 0.75$ is enough.

On the other hand, little connectivity can give rise to stable states which are not all-in-phase. As an example, it is well known that rings of any size support a twisted state, sometimes called the Mexican wave; this state can be easily constructed explicitly and it is straightforward to prove its stability, meaning that there is a set of initial conditions with positive Lebesgue measure that converge to the twisted state. This phenomenology can be extended to Wiley-Strogatz-Girvan (WSG) networks \cite{wiley2006size} that are constructed by linking each node on a cycle of $n$ nodes with its $k$-nearest neighbors, for $k\le 0.34n$ \cite{ling2019landscape}.

Furthermore, Canale and Monz\'on \cite{canale2008almost, canale2015exotic} constructed a sequence of networks from the lexicographic product of a WSG network together with complete graphs to exhibit networks that are highly connected ($\mu = (15/22)^-)$ but do not globally synchronize. This was improved in \cite{yoneda2021lower} to $0.6838\dots$

From a different perspective, \cite{bhamidi2019weakly, coppini2020law, delattre2016note, luccon2020quenched} give a rigorous mathematical treatment to obtain the mean-field limit of the system when the number of oscillators tend to infinity in a stochastic situation. In \cite{oliveira2018interacting, coppini2022long} the authors are able to push this analysis for times large enough to obtain large deviations results and conclusions about synchronization (which does not follow directly from the scaling limit in compact time intervals). See also \cite{chiba2016mean} for scaling limits in a deterministic setting and in a different regime \cite{medvedev2018continuum}. Related models have been considered in the Ph.D. Thesis by Yu \cite{HBY}.

We remark that there is an extensive literature studying Kuramoto model rigorously from a probabilistic perspective. A non-extensive list includes \cite{bertini2010dynamical, bertini2014synchronization, giacomin2012global, ling2019landscape, ling2020critical}. Although we have a different setting in mind, we share with them that standpoint.

The general picture to understand whether for a given graph, a system of identical Kuramoto oscillators synchronizes or admits a twisted stable equilibrium is far from being completely understood. Despite some light shed by these remarkable articles we still do not have an answer for the majority of graphs below the connectivity threshold.

Moreover, up to our knowledge, all the examples constructed to find sparse graphs that do synchronize and dense graphs that do not are not generic in the sense that the results obtained in most of the cases have not been shown to be robust under small perturbations of the graphs. 

In view of this, it is important to understand this issue for different models of random graphs since in these cases one obtains results that hold with high probability and hence, they represent generic behaviors. Those are the ones that we can observe in nature. 

In particular, it is interesting to understand the behavior of the system for large times in time-dependent networks since they provide, among other things, a framework to consider small perturbations of a given graph by allowing the edges to disappear for short time intervals. They also include the possibility to model networks that evolve with time, which is the case in many real world phenomena. A natural question that pops up in this context is weather the fact that the network changes with time leads to global synchronization.

The mathematical study of synchronization in Kuramoto model on random graphs can be traced back to \cite{ling2019landscape}, where the authors proved that for highly connected Erd\H{o}s-R\'enyi graphs we do have global synchronization with high probability. The threshold for synchronization obtained in \cite{ling2019landscape} is $p \ge Cn^{-1/3}\log n$. This has been improved in \cite{kassabov2022global} to $p \gg \log^2n/n$ and later on in \cite{Abdalla2022} to $p \ge c \log n/n$, with $c>1$, which coincides with the connectivity threshold $\log n/n$.


In this note we propose a Kuramoto model of coupled oscillators on a random graph that changes with time according to a Markovian dynamics. As we will see, one of the possible dynamics that we will consider is inspired by the representation of fireflies synchronizing their flashings (as observed e.g. in Southeast Asia) while changing their locations, but the introduction of dynamic topologies is interesting for many other phenomena in which the interactions between different oscillators are not constant in time. In fact, a large number of systems in physics, biology, engineering and social science exhibit a time-varying structure \cite{Moreau, Multi_Review}. 

Synchronization in systems of mobile agents has been widely considered in the physics literature (see \cite{Multi_Review} and references therein), inspired mainly by transportation and mobility systems, mobile phone networks, multi-agent robotics, and epidemic modeling.

Despite of the large body of literature available, rigorous mathematical results regarding the Kuramoto model on \emph{dynamic} graphs -that is when the topology of the underlying graph changes with time- seem to be absent.

We remark that both the introduction of randomness and the time-dependent connections complicate substantially the stability analysis. Obtaining optimal results for fixed velocities seems a rather hard issue which would probably require new technical tools; as far as this manuscript is concern, these issues are left open as interesting research directions. Our contributions, in turn, are related to the limiting case in which the velocity tends to infinity. In this scenario, one can manage to combine ideas from probability theory, specifically Markov Chains and Averaging Principles, together with stability analysis for the standard Kuramoto model to obtain synchronization results either for any velocity and close enough initial configuration (Theorem \ref{thm:basin}) or for high enough velocity, i.e. $\ep$ small enough (Theorem \ref{thm-main}).

Finally, we note the connection with the classical XY-model \cite{maes2011rotating} and that for other famous models of synchronization (sometimes called consensus or flocking) such as Cucker-Smale model \cite{CuSm} or Vicsek model \cite{Vicsek}, time-varying topologies have been mathematically explored from different perspectives, generally in a different direction than ours. See e.g. \cite{jadbabaie2003coordination, Stilwell,  bonnet2021consensus} and references therein.

Also intriguing is the precise relation of our class of models with the deep theory of {\it{random attractors}} in Random Dynamical Systems \cite{arnold1995random, crauel1994attractors, flandoli2017synchronization}, in particular the synchronization by noise phenomena.

\subsection{Model description and main results}

Let $\G=(V,E)$ be a finite connected graph with vertex set $V$ and edge set $E$ that we will call the \textit{skeleton}. Assume that $\G$ is endowed with edge weights, i.e. to each pair $(u,v)\in V\times V$ is assigned a nonnegative symmetric weight 
\[
\pi(u,v)=\pi(v,u)\in[0,\infty).
\]
We thus can identify the edge set with the pairs of vertices with positive weights:
\[
E=\big\{(u,v)\in V^2:\; \pi(u,v)>0\big\}.
\]

For any $u\in V$, we denote its vertex degree by 
\begin{align}\label{bd-deg}
\pi(u):=\sum_{v\in V}\pi(u,v)>0;
 \qquad \text{and} \qquad A:= \sup_{u\in V}\pi(u) <\infty.
\end{align}

We will consider a general model that imposes Markovian dynamics on the graph and the Kuramoto \abbr{ODE} system running on top of it. Each of these dynamics can have, in principle, its own scale. Many interesting questions arise naturally in terms of the possible emergent phenomena that pop up from the interplay between the two different scales and the features of the particular Markovian dyanmics on the graph. To keep focused we will center our work in two specific examples: a family of $N$ independent random walkers on the graph and what we call the \textit{Dynamic Random Conductance Model} (\abbr{DRC}). Before we state them, we highlight that the underlying \emph{finite} graph $\G$ is fixed throughout the manuscript, as well as the size $N$ of the population of random walkers. 

\subsubsection*{Kuramoto's random walkers}
For the first one, given an integer $N\ge2$ and $\ep\in(0,\infty)$, let $\bX^\ep(t):=(X^\ep_1(t),...,X^\ep_N(t))$ be a collection of $N$ independent continuous time (constant speed) simple random walks on the graph $\G$ with constant jump rate $\ep^{-1}$. Since $\G$ is finite, each $X^\ep_i(t)$ is an irreducible, positive recurrent Markov chain, and has a unique reversible measure  independent of $\ep$
\begin{align}\label{inv-meas}
\mu(u)=\frac{\pi(u)}{\sum_{v\in V}\pi(v)}, \quad u\in V.
\end{align}


\begin{remark}\label{rmk:generalization}
We can think of $\bX^\ep(t)$ as a subgraph of $\G$ that changes with time and interpret this dynamics as a Markov chain with state space the set of subgraphs of $\mathbb G$. We point out that the main requirement on the process $\bX^\ep(t)$ that we will use in our arguments is the fact that any connected subgraph of $\G$ with $\le N$ vertices is positive recurrent for the dynamics. In consequence, our results could apply to more general processes such as other types of random walks, exclusion processes or zero-range processes, among others. 

Moreover, we will show that in the regime $\ep \to 0$ and $t\to\infty$ the particular graph structure plays no role as far it is connected. This is because the family of subgraphs $\G$ defined by the random walks is rich enough to guarantee that the oscillators work on the high density regime for large enough time. In particular the complete graph in $N$ nodes is positive recurrent.

\end{remark}

Finally, we will denote by $\T^N:=\left({\R}/{2\pi\Z}\right)^N$ the $N$-dimensional torus, naturally identified with $(\mathbb{S}^1)^N$, where $\mathbb{S}^1$ is the unit circle and by $\|\cdot\|_2$ and $\|\cdot\|_\infty$ the norms in finite dimensional $\ell^2$ and $\ell^\infty$ respectively, that is
\[
\|\pmb{x}\|_2=\left(\sum_{i=1}^N|x_i|^2\right)^{1/2}, \quad \|\pmb{x}\|_\infty=\max\{|x_1|,\ldots,|x_N|\}.
\]

\medskip

A {\it{Kuramoto Random Walk}} (\abbr{KRW} for short) with skeleton $\G$, natural frequencies $\pmb{\omega}:=(\omega_i)_{i=1}^N\in\T^N$, and coupling constant $K>0$ is an \abbr{ODE} system for $\pmb{\theta}^\ep(t):=(\theta^\ep_1(t),...,\theta^\ep_N(t))\in\T^N$ with random time-dependent coefficients of the form
\begin{equation}\label{dyn-km}
 \begin{aligned}
\frac{\displaystyle d}{\displaystyle dt}\theta^\ep_i(t) & =  \omega_i+\frac{\displaystyle K}{\displaystyle N}\displaystyle\sum_{j=1}^Na_{ij}(\bX^\ep(t))\sin\big(\theta^\ep_j(t)-\theta^\ep_i(t)\big),  \quad i=1,...,N,\; t\ge0,\\
\pmb{\theta}^\ep|_{t=0}&= \pmb{\theta}^\ep_0,  \quad \bX^\ep|_{t=0}= \bx^\ep_0,&&
\end{aligned}
\end{equation}
where for $\bx=(x_1,...,x_N)\in V^N$, $i,j=1,...N$, we write
\begin{align}\label{dyn-weight}
a_{ij}(\bx):=\pi(x_i,x_j)\ge0.
\end{align}
That is, the strengh of the link between $\theta_i$ and $\theta_j$ is given by $\pi(x_i,x_j)$ if $x_i$ and $x_j$ occupy neighboring sites and zero otherwise. We call $\theta_i(t)\in\T$ the {\it{phase}} of the $i$-th oscillator at time $t$. We will often abbreviate using the notation
\begin{align}\label{km-drift}
b_i(\pmb{\theta},\bx):=\omega_i+\frac{K}{N}\sum_{j=1}^Na_{ij}(\bx)\sin\big(\theta_j-\theta_i\big).
\end{align}

One should think of each mobile agent $i\in\{1,...,N\}$ as occupying position $X^\ep_i(t)\in\G$, moving at speed $\ep^{-1}$ and carrying a phase $\theta^\ep_i(t)\in\T$. Each pair of agents $(i,j)$ tends to synchronize their phases (or flashings) only when they are adjacent with respect to the underlying graph structure.

It is natural to assume that all vertices have a self edge (or loop), i.e.
\[
\pi(u,u)>0, \quad \forall u\in V,
\]
so that two agents occupying the same vertex will surely interact and we'll do so when considering this model.

\begin{remark}\label{rmk:bdd}\leavevmode
\begin{enumerate}[label=(\alph*)]

\item Since $a_{ij}(\bx)\le A$ for any $\bx\in V^N$ and $i,j$,
\[
|b_i(\pmb{\theta},\bx)|\le KA+\|\pmb{\omega}\|_\infty\text{ for any }\pmb{\theta},\bx, i.
\] 
\item We have $a_{ij}(\bx)=a_{ji}(\bx)$ for any $\bx\in V^N$ and $i,j$.
\end{enumerate}
\end{remark}

{We mainly consider homogeneous models, where the natural frequencies $\omega_i=\ovl \omega$ for all $i$.
It is well-known that in this case, by rotating each $\theta_i(t)$ by the same $\ovl\omega t$, we can assume without loss of generality that the common frequency $\ovl\omega =0$.} In such case, summing \eqref{dyn-km} in $i$ we see that, due to the fact that  $a_{ij}(\bX^\ep(t))$ is symmetric and that sine is an odd function, the sum $t\mapsto\sum_{i=1}^N\theta^\ep_i(t)$ is constant in time. Up to a global translation of the initial condition, one can impose without loss of generality that $\sum_{i=1}^N\theta^\ep_i(0)=0$ such that $\mathbf{0}:=0\mathbf{1}_N$ is the all-in-phase state and a fixed point of the dynamics, where $\mathbf{1}_N$ denotes the all $1$ vector of dimension $N$. We are mainly interested in whether the dynamic Kuramoto models \eqref{dyn-km} and \eqref{kuramoto.RCM} below achieve \emph{asymptotic phase synchronization}, namely, whether $\pmb{\theta}^\ep(t)\to\mathbf{0}$ for fixed $\ep$ as $t\to\infty$,  with high probability.

One could be tempted to use a change of variables to extrapolate the results from the smallness regime to an arbitrary velocity. We point out, however, that this cannot be achieved because of the lack of scale invariance in the bound of the weights. 
Indeed, it is well-known that the laws of the random walks $\{\bX^\ep(t)\}$ for different $\ep$ are related by a time change:
\begin{align*}
\bX^\ep(\ep t)=\bX^1(t), \quad t\ge0.
\end{align*}
Furthermore, a direct computation then shows that the laws of the oscillators are related by
\[
\pmb\theta^\ep(\ep t)=\wt{\pmb\theta^1}(t), \quad t\ge 0,
\]
where $\{\wt{\pmb\theta^1}(t)\}$ is not the solution of \eqref{dyn-km} with coefficients $a_{ij}(\bX^1(t))$, but is instead with coefficients $\ep\, a_{ij}(\bX^1(t))$. Since the set up of our model is such that the weights $\pi(u,v)$ (on which $a_{ij}(\cdot)$ are based) have to be given apriori and satisfy the uniform bound \eqref{bd-deg}, this connection between $\{\pmb\theta^\ep(t)\}$ for different $\ep$'s is not useful to transfer what we obtained for \eqref{dyn-km} at $\ep$ small enough to the case $\ep=1$. Namely, below we'll give a precise description of the behavior of solutions of \eqref{dyn-km} as $\ep \to 0$, in the sense that our synchronization results work for sufficiently small (positive) values of $\ep$. This description can not be exported to other values of $\ep$. Although we will give some information for fixed $\ep$, several questions remain an open problem.

\begin{remark}[The role of the parameter $K$]\label{rmk:K}
The coupling constant $K$ does not play a significant role in our case and it does not have a significant effect on the results we discuss. When dealing with frequency synchronization, there is a critical value $K_c$ below which synchronization fails, see \cite{strogatz2000kuramoto} for further details. The condition for such failure is given in terms of the size of $\omega_i$ relative to $K$. However, we deal chiefly with phase synchronization, in which case it is a necessary condition that all natural frequencies coincide (it is easy to see that phase synchronization cannot be achieved otherwise). Then, as pointed out above, one may assume all natural frequencies are equal to $0$ and the parameter $K$ ceases to be relevant.
\end{remark}

\subsubsection*{Dynamic Random Conductance Model.}
The second model we consider in this paper consists in allowing each of the edges in $\mathbb G$ to appear and disappear at some rate for each other, independently. More precisely, let $\epsilon, \kappa >0$. For each $t\ge 0$, we consider the subgraph of $\mathbb G$ given by $G^{\ep,\kappa}(t) = (V,\cE^{\ep,\kappa}(t))$. For each edge $e \in E$ we have a continuous time Markov chain $(e^{\ep,\kappa}(t))_{t\ge 0}$ with state space $\{0,1\}$; the chain jumps from one to zero at rate $\epsilon^{-1}$ 
and from zero to one at rate $\kappa$. Each of the chains evolves independently. The edge set $\cE^{\ep,\kappa}(t)$ is given by
\[
\cE^{\ep,\kappa}(t):= \{ e \in E \colon e^{\ep,\kappa}(t)=1 \}.
\]
Unless we state it, we will assume $e(0)=1$ for every $e \in E$, so that $\epsilon=\infty$ corresponds to the frozen graph $\mathbb G$, while $\epsilon, \kappa\in(0,\infty)$ allows each edge of the graph to be present or absent at different times, independently.

The \textit{Kuramoto Model on the Dynamic Random Conductance Model} with skeleton $\G$, natural frequencies $\pmb{\omega}:=(\omega_u)_{u\in V}\in\T^{|V|}$, and coupling constant $K>0$ is the \abbr{ODE} system with random coefficients
\begin{equation}
\label{kuramoto.RCM}
\begin{aligned}
\frac{\displaystyle d}{\displaystyle dt} {\theta}^{\ep,\kappa}_u(t) = \, & \displaystyle \omega_u+\frac{\displaystyle K}{\displaystyle{|V|}}\sum_{v\in V}a_{uv}(\cE^{\ep,\kappa}(t)) \sin({\theta}^{\ep,\kappa}_v(t)-{\theta}^{\ep,\kappa}_u(t)), \quad u\in V,\; t>0\\
{\pmb\theta}^{\ep,\kappa}|_{t=0}  = \, &{\pmb\theta}^{\ep,\kappa}_0, \quad \cE^{\ep,\kappa}|_{t=0}=\cE^{\ep,\kappa}_0,
\end{aligned}
\end{equation}
where for any $\cE\subset E$, $(u,v)\in E$ we denote
\[
a_{uv}(\cE):=1_{ \{(u,v)\in \cE\}} \pi(u,v)\ge0.
\]
We abbreviate the \abbr{RHS} of \eqref{kuramoto.RCM} using
\begin{align*}
b_u(\pmb\theta,\cE):=\omega_u+\frac{K}{|V|}\sum_{v\in V}a_{uv}(\cE) \sin(\theta_v-\theta_u), \quad u\in V, \; \pmb\theta\in\T^{|V|},\;\cE\subset E.
\end{align*}
It has a similar form as \eqref{km-drift}, having coefficients $a_{uv}(\cE)$ instead of $a_{ij}(\bx)$, and $|V|$ instead of $N$. Thus, we use a similar notation to allow for a uniform treatment of the two models.

The name dynamic random conductance model comes from the famous (static) random conductance model \cite{biskup2011recent, kumagai2014random} that allows each edge of the graph to have a random weight. This can be done also in this context but for simplicity we prefer to allow each edge $\{u,v\}$ to take just two values: zero or $\pi(u,v)$.

Observe that $G^{\ep,\kappa}(t)=(V,\cE^{\ep,\kappa}(t))_{t\ge 0}$ is, as in the previous model, a continuous time Markov chain with state space the subgraphs of $\mathbb G$.

We are particularly interested in this model when $\mathbb G$ is a Wiley-Strogatz-Girvan graph since those networks allow twisted stable states. We will show that for every $\epsilon, \kappa \in(0,\infty)$ these states are no longer equilibria of the system. Hence, allowing the edges to disappear for a tiny amount of time, eliminates the twisted stable states. In this sense one can understand those twisted state as \textit{spurious}. Moreover, we will show that the only stable equilibrium is the all-in-phase state.

Nevertheless, twisted states have been observed in real phenomena in the form of propagating waves and hence we wouldn't qualify these states and their stability as spurious, but we think that in order to state their existence as an emergent behavior of the model we should be able to prove their presence as a more persistent object.

Both the \abbr{KRW} and \abbr{DRC} are special cases of a model given by the random \abbr{ODE} \eqref{kuramoto.RCM} in which $G(t)=(V(t),E(t))$ is a continuous time Markov chain in the space of subgraphs of $\mathbb G$.

{
\begin{remark}
For the \abbr{ODE}s \eqref{dyn-km} and \eqref{kuramoto.RCM}, existence and uniqueness (for every realization of the randomness) hold by Carath\'eodory's existence theorem and Remark \ref{rmk:bdd} (a) and Lemma \ref{rmk:lip} (b). 
\end{remark}}


\medskip

\subsection{Main results.}

Our goal is to understand how the fact that the network evolves with time in a random way affects the phase synchronization of the system. The general picture to have in mind is that randomness leads to global synchronization since the specific graphs that support other stable states and the initial conditions that converge towards those states are not persistent. Hence, at some point the system will reach the domain of attraction of the all-in-phase state.
More precisely, although the topology of the network changes with time, for a large amount of time there will be a tendency towards the phase-cohesive set $\Delta(\gamma)$, which is defined by
\begin{align}\label{basin}
\Delta(\gamma):=\{\pmb{\theta}:\, \max_{i,j}|\theta_i-\theta_j|\le \gamma\}.
\end{align}
(Here for $\theta_i,\theta_j\in\T$ we write, with a slight abuse of notation, $|\theta_i-\theta_j|\in[0,\pi]$ for their geodesic distance, see \cite[p. 259]{bullo2020lectures}.)

In Section \ref{sec.anyep} we will show that for $\gamma < \pi/2,\:\Delta(\gamma)$ is invariant for the dynamics independently of the network topology and hence the all-in-phase state is achieved almost surely starting from $\Delta(\gamma)$. As a consequence, we have the following theorem.
\begin{theorem}\label{thm:basin}
Fix any $N\ge 2$ and $\gamma\in(0,\pi/2)$, and consider $\pmb\omega=\mathbf0$ in \eqref{dyn-km} (resp. \eqref{kuramoto.RCM}). Then for any $\ep>0$ (resp. $\ep,\kappa>0$), and any initial condition $\pmb{\theta}^\ep_0$ (resp. ${\pmb\theta}^{\ep,\kappa}_0$) $\in \Delta(\gamma)$ the system converges $\P$-a.s. to $\mathbf{0}$ (up to a global translation) as $t\rightarrow\infty$.
\end{theorem}

We will further characterize the equilibria of the system and show that $\mathbf{0}$ is asymptotically stable while all other equilibria are unstable; due to randomness, we have non-autonomous \abbr{ODE}s and hence we first need to clarify what we mean with that. We give the following definition:

\begin{definition}\label{def:equil}
For \abbr{KRW}, we say that $\wh{\pmb\theta^\ep}\in\T^N$ is a fixed point (or equilibrium point) of the \abbr{ODE} \eqref{dyn-km} if $\P$-a.s.
\begin{align*}
b_i\big(\wh{\pmb\theta^\ep}, \bX^\ep(t)\big)=0, \quad \forall t\ge 0,\; i=1,...,N.
\end{align*}
For \abbr{DRC}, we say that $\wh{\pmb\theta^{\ep,\kappa}}\in\T^{|V|}$ is a fixed point (or equilibrium point) of the \abbr{ODE} \eqref{kuramoto.RCM} if $\P$-a.s.
\begin{align*}
b_u\big(\wh{\pmb\theta^{\ep,\kappa}}, \cE^{\ep,\kappa}(t)\big)=0, \quad \forall t\ge 0,\; u\in V.
\end{align*}
\end{definition}

The characterization of the possible equilibrium states, that follows from the irreducibility of the processes, is now enclosed in the following proposition.
\begin{proposition}\label{ppn:char-equil}
(a) For \eqref{dyn-km} and any $\ep\in(0,\infty)$, $\P$-a.s. any fixed point $\wh{\pmb\theta^\ep}$ belongs to the deterministic set 
\begin{align*}
\bS_1:=\big\{\pmb\theta\in\T^N:\; b_i(\pmb\theta, \bx)=0, \; \forall \bx\in V^N, \, i=1,...,N\big\}.
\end{align*}
For \eqref{kuramoto.RCM} and any $\ep,\kappa\in(0,\infty)$, $\P$-a.s. any fixed point $\wh{\pmb\theta^{\ep,\kappa}}$ belongs to the deterministic set 
\begin{align*}
\bS_2:=\big\{\pmb\theta\in\T^{|V|}:\; b_u(\pmb\theta, \cE)=0, \; \forall \cE\subset E, \, u\in V\big\}.
\end{align*}
(b) Now restrict to the homogeneous case, i.e. $\pmb\omega=\mathbf0$. For \abbr{KRW}, assume that the skeleton graph $\G$ is not complete\footnote{If $\G$ is a complete graph, this model reduces to static Kuramoto on a complete graph of $N$ vertices. In this case there is no dynamics for the graph and not randomness at all. As a consequence, the behavior of the model is well understood.}, i.e. there exists a pair of vertices $u^*\neq  v^*\in V$ with $\pi(u^*,v^*)=0$, then
\begin{align*}
\bS_1=\bS:=\big\{\pmb\theta\in \T^N:\; |\theta_i-\theta_j|\in\{0,\pi\}, \; \forall i,j=1,...,N\big\}.
\end{align*}
For \abbr{DRC}, we have that
\begin{align*}
\bS_2=\bS:=\big\{\pmb\theta\in \T^{|V|}:\; |\theta_u-\theta_v|\in\{0,\pi\}, \; \forall u,v\in V\big\}.
\end{align*}
\end{proposition}

We also give some results in the direction of stability, see Proposition \ref{prop:stab}. We believe that a much stronger conclusion should hold, and indeed we conjecture that the state $\mathbf0$ attracts the whole phase space except a set of Lebesgue measure zero, see Conjecture \ref{conj:attractor}.

A result in the direction of Conjecture \ref{conj:attractor} is global synchronization with high probability. This will be proved with the aid of an averaging principle for \eqref{dyn-km} for small values of $\epsilon$. Indeed, in Section \ref{sec.aver} we prove that as $\ep\rightarrow0^+$ solutions of \eqref{dyn-km} converge to solutions of an appropriate averaged -deterministic- \abbr{ODE}. More precisely, consider the system $\ovl{\pmb{\theta}}(t):=(\ovl{\theta}_1(t),...,\ovl{\theta}_N(t))\in\T^N$ with constant, deterministic coefficients and the same natural frequencies $\pmb{\omega}:=(\omega_i)_{i=1}^N\in\T^N$ as in \eqref{dyn-km}
\begin{equation}\label{ave-km}
\begin{cases}
\frac{\displaystyle d}{\displaystyle dt}\ovl{\theta}_i(t)=\omega_i+\frac{\displaystyle K}{\displaystyle N}\sum_{j=1}^N\ovl{a}\sin\big(\ovl{\theta}_j(t)-\ovl{\theta}_i(t)\big), \quad i=1,...,N, \; t\ge0\\\\
\ovl{\pmb{\theta}}|_{t=0}=\ovl{\pmb\theta}_0, 
\end{cases}
\end{equation}
where 
\begin{align*}
\ovl{a}:=\sum_{u,v\in V}\pi(u,v)\mu(u)\mu(v)>0.
\end{align*}
As before, we will abbreviate the \abbr{RHS} of \eqref{ave-km} using
\[
\ovl{b}_i(\pmb\theta):=\omega_i+\frac{K}{N}\sum_{j=1}^N\ovl{a}\sin\big(\theta_j-\theta_i\big).
\] 
Then, we have the following proposition, in it's statement we write $\pmb{\theta}^{\ep,\pmb{\alpha}}(t),\ovl{\pmb{\theta}^{\pmb{\alpha}}}(t)$ to indicate the initial condition of respective \abbr{ODE} \eqref{dyn-km}, \eqref{ave-km} is $\pmb{\alpha}\in\T^N$.
\begin{proposition}\label{prop:ave}
Consider the model \eqref{dyn-km}. Fix $N\ge2$ and consider any $\pmb{\omega}=\{\omega_i\}_{i=1}^N\in\T^N$. Then, for any finite $T$ we have that
\begin{align}\label{ave-rate}
\lim_{\ep\to0}\sup_{\pmb{\alpha}\in\T^N}\E\big[\sup_{t\in[0,T]}\|\pmb{\theta}^{\ep,\pmb{\alpha}}(t)-\ovl{\pmb{\theta}^{\pmb{\alpha}}}(t)\|_\infty\big]=0.
\end{align}
\end{proposition}

With the averaging principle at hand, in Section \ref{sec.global} we prove the following theorem for \abbr{KRW} which, roughly speaking, says that for any compact set $\Omega'$ inside a subset of $\T^N$ of full Lebesgue measure, convergence to $\mathbf{0}$ as $t\rightarrow\infty$ holds for $\ep$ small enough (depending on $\Omega'$) and with high probability.
\begin{theorem}\label{thm-main}
Consider \abbr{KRW} \eqref{dyn-km}. Fix $N\ge2$ and consider $\pmb\omega=\mathbf0$. There exists a set $\Omega\subset\T^N$ of full Lebesgue measure such that for any compact subset $\Omega'\subset\joinrel\subset\Omega$ and any $\eta>0$, there exists $\ep_0=\ep_0(\Omega',\eta)$, such that for any $\ep\in(0,\ep_0)$ and $\pmb{\alpha}\in\Omega'$, with probability at least $1-\eta$ the unique solution of \eqref{dyn-km} initialized at $\pmb{\theta}^{\ep}(0)=\pmb{\alpha}$ converges to $\mathbf{0}$ (up to a global translation) as $t\to\infty$.
\end{theorem}

\subsection{Short discussion on the physics literature}
Synchronization in time-varying networks have been widely considered in the physics community \cite{Faggian2019, Multi_Review, Buscarino2015, Fujiwara2011, Uriu2013, Frasca2008}. They consider Kuramoto as well as other synchronization models, some of which carry a similar flavor as our models. In particular, the averaging effect due to fast changes of network topologies has been well noted.
We remark that considering time-varying topologies leads to non-autonomous dynamical systems and hence an increasing difficulty to obtain rigorous results for their behavior. Our paper seems to be the first one that is mathematically rigorous, and also puts forward a sufficiently general framework to study related phenomena.

In \cite{Faggian2019}, the authors study the Kuramoto model on an Erd\"os-Renyi random graph $\cG(|V|,p)$ whose connectivity is resampled every $T$ units of time (called ``re-wiring''). Up to minor details, our \abbr{DRCM} model can be thought as a generalization of that model that obtains the one considered in \cite{Faggian2019} as a particular case when the skeleton is the complete graph, $\ep^{-1}=(1-p)T^{-1}$ and $\kappa=pT^{-1}$, with $p$ allowed to depend on $|V|$. In \cite{Faggian2019} the authors note, but do not rigorous prove, an averaging principle that leads, in the fast rewiring limit, to an all-to-all Kuramoto model and made the conclusion that fast enough rewiring facilitates synchronization. We note that their claim encompasses the case with inhomogeneous natural frequencies $\omega_i$, which are assumed to be i.i.d. mean-zero Gaussian, a situation not covered by our manuscript. We remark that obtaining conclusions about the stability of an equilibrium point of a system from the averaged limit equation is delicate since averaging principles hold in finite time intervals, while the stability of equilibrium points has to do intrinsically with the limit as time goes to infinity. Due to this fact we obtain results about synchronization from the averaged equation only in the homogeneous case, which is more tractable from a mathematical point of view.

In \cite{Uriu2013}, the authors study Kuramoto oscillators (``mobile agents'') on a discrete torus, where each agent is only connected to those others within Euclidean distance $r>0$. Time-dependence enters when with rate $\lambda>0$, each agent independently switches its spatial position with one of its neighbors. Making connections to continuum percolation, the authors note a phase transition between several different regimes in terms of the parameters $r,\lambda$ and the coupling constant $K$ between pairs of oscillators. Our \abbr{KRM} model contains this situation as a special case in which the skeleton is the graph obtained by linking each vertex of the discrete torus to others within distance $r$ (though the random walk mechanism is slightly different, cf. Remark \ref{rmk:generalization}).

Similarly, in \cite{Fujiwara2011} the authors consider a model in which the spatial motion of the agents follows Newtonian mechanics with their headings (i.e. the direction of their velocities) independently and uniformly resamplied every $\tau_M$ units of time. The connectivity of the oscillators is still dictated by spatial proximity. Remarkably, their emphasis is to argue that slow change of the network topologies increases the time to achieve synchronization. This corresponds to the regime of fixed (and not small) $\ep$ in our \abbr{KRM} model, whose precise behavior remains a conjecture for us.

More advanced studies also focus on the effect of (usually fast) time-varying topologies on the formation of more complicated stable patterns for Kuramoto model, such as chimera states \cite{Buscarino2015}. To conclude, these physics works certainly provide a wealth of inspirations for any future mathematically rigorous investigation of dynamic Kuramoto and related models.

\medskip

\section{Synchronization for fixed (arbitrary) $\ep>0$.}\label{sec.anyep}

We start this section with the proof of Theorem \ref{thm:basin}, that gives phase synchronization for a dynamic starting on the phase-cohesive set $\Delta(\gamma)$. {It applies to both \abbr{KRW} \eqref{dyn-km} and \abbr{DRC} \eqref{kuramoto.RCM}.} The basic idea is to exploit the fact that in $\Delta(\gamma)$ there is a uniform spectral gap for the Hessian of the potential that generates the dynamics as well as the fact that mentioned above that $\Delta(\gamma)$ is invariant for the dynamics disregarding of network topology. 

\begin{proof}[Proof of Theorem \ref{thm:basin}]
{We write the proof for \abbr{KRW}, and comment at the end for \abbr{DRC}.}
First we note the set $\Delta(\gamma)$ is convex, as can be easily checked from its definition \eqref{basin}. Next, we make the following claim: 
\[
\pmb{\theta}^\ep(0)\in\Delta(\gamma)\quad\Rightarrow\quad \pmb{\theta}^\ep(t)\in\Delta(\gamma)\text{ for all }t\ge 0.
\]
Indeed, we can argue as in \cite[Eq. $(4.9)$]{ling2020critical, dorfler2011critical}: fix any $t\ge 0$, assume that $\pmb{\theta}^\ep(t)\in\Delta(\gamma)$ holds, and let $i_{max}$, $i_{min}\in\{1,...,N\}$ be two indices such that
\[
|\theta^\ep_{i_{max}}(t)-\theta^\ep_{i_{min}}(t)|=\max_{i,j}|\theta^\ep_i(t)-\theta^\ep_j(t)|,
\]
where we recall that $|\cdot|$ denotes geodesic distance. Without loss of generality, we assume 
\[
\sin\big(\theta^\ep_{i_{max}}(t)-\theta^\ep_{i_{min}}(t)\big)\ge0,
\]
which implies that for any $j$
\[
\sin\big(\theta^\ep_{i_{max}}(t)-\theta^\ep_{j}(t)\big)\ge0, \quad 
\sin\big(\theta^\ep_{j}(t)-\theta^\ep_{i_{min}}(t)\big)\ge0.
\]
Now we have by \eqref{dyn-km} with $\pmb\omega=0$,
\begin{align*}
&\frac{d}{dt}\big(\theta^\ep_{i_{max}}(t)-\theta^\ep_{i_{min}}(t)\big)\\
&=-\sum_{j=1}^Na_{ij}(\bX^\ep(t))\big[\sin\big(\theta^\ep_{i_{max}}(t)-\theta^\ep_{j}(t)\big)+\sin\big(\theta^\ep_{j}(t)-\theta^\ep_{i_{min}}(t)\big)\big]\le 0,
\end{align*}
which yields our claim. 

Next, we rewrite \eqref{dyn-km} as a gradient flow
\begin{align}\label{grad-fl}
\frac{d}{dt} \theta^\ep_i(t)&=-\frac{K}{N}\partial_{\theta_i}\cU(\pmb{\theta}^\ep(t),\bX^\ep(t)), \quad i=1,...,N
\end{align}
where we denote the energy function
\begin{align}\label{energy}
\cU(\pmb{\theta},\bx):=\frac{1}{2}\sum_{i, j=1}^Na_{ij}(\bx)\big[1-\cos\big(\theta_j-\theta_i\big)\big].
\end{align}
The Hessian of $\cU$ with respect to $\pmb{\theta}$ has entries
\begin{align}\label{hessian}
&\partial^2_{\theta_i\theta_j}\cU(\pmb{\theta},\bx)\nonumber\\
&=-a_{ij}(\bx)\cos\big(\theta_j-\theta_i\big)1_{\{j\neq i\}}+\sum_{k\neq i}a_{ik}(\bx)\cos\big(\theta_k-\theta_i\big)1_{\{j=i\}}, \quad i,j=1,...N.
\end{align}

When $\pmb{\theta}\in\Delta(\gamma)$, the matrix \eqref{hessian} can be seen as the graph Laplacian of a fictitious graph $\cG^\bx$ on $N$ vertices $\{1,...,N\}$ with nonnegative symmetric weights $\wt{\pi}(i,j):=a_{ij}(\bx)\cos\big(\theta_j-\theta_i\big)$ between $i$ and $j$, hence it is always non-negative definite. Its smallest eigenvalue is $0$, with corresponding eigenvector $\pmb{1}_N$. In the case that $\cG^\bx$ is a connected graph, the second smallest eigenvalue of \eqref{hessian} is strictly positive, a well-known fact in graph theory, see e.g. \cite[Chapter 6]{bullo2020lectures}. Since there are only finitely many connected subgraphs of $\G$, 
and $\cos\big(\theta_j-\theta_i\big)\ge \cos\gamma>0$ in $\Delta(\gamma)$, the spectral gap of the Hessian can be taken to be uniformly positive whenever the fictitious graph $\cG^\bx$ is connected. We denote by $\wt c=\wt c(\gamma, \G)$ this spectral gap.

It now suffices to notice that the amount of time $t\in[0,\infty)$ when the fictitious graph $\cG^{\bX^\ep(t)}$ on $\{1,...,N\}$ induced by $\bX^\ep(t)$ is connected has infinite Lebesgue measure, $\P$-a.s. for any $\ep\in(0,\infty)$. Indeed, if we denote by $\cH$ the set of all connected subgraphs of $\G$, 
then $\cH$ is a positive recurrent set for $\bX^\ep(t)$ (here we mean the vertex set of each element of $\cH$). Since the jump rate of  $\bX^\ep(t)$ is constant $N\ep^{-1}$, various notions of recurrence are equivalent, and we have that 
\[
\int_0^\infty 1\{\bX^\ep(t)\in\cH\}dt=\infty, \quad \P\text{-a.s.}
\]
Clearly whenever $\{\bX^\ep(t)\in\cH\}$ happens, the fictitious graph $\cG^{\bX^\ep(t)}$ is connected.


Now we argue that although the energy landscape $\cU(\pmb{\theta},\bX^\ep(t))$ is time-varying and random, due to the uniform positive spectral gap of the Hessian of $\cU(\pmb{\theta},\bX^\ep(t))$ whenever $t$ is a time for which the fictitious graph $\cG^{\bX^\ep(t)}$ discussed above is connected, the dynamic will still converge to $\mathbf{0}$ as $t\to\infty$, provided initialized within $\Delta(\gamma)$. Here, without loss of generality we assumed that $\sum_{i=1}^N\theta^\ep_i(t)\equiv 0$ hence $\mathbf{0}$ is a fixed point of the dynamic and $\pmb{\theta}^\ep(t)$ is orthogonal to the eigenspace spanned by $\mathbf{1}_N$. We compute, using \eqref{grad-fl}
\begin{equation}\label{eq.dttheta}
\frac{1}{2}\frac{d}{dt}\|\pmb{\theta}^\ep(t)\|_2^2=\big\langle\pmb{\theta}^\ep(t), \frac{d}{dt}\pmb{\theta}^\ep(t)\big\rangle=-\frac{K}{N}\big\langle\pmb{\theta}^\ep(t), \nabla_{\pmb{\theta}}\cU(\pmb{\theta}^\ep(t),\bX^\ep(t))\big\rangle,
\end{equation}
where $\|\cdot\|_{p}:=\|\cdot\|_{\ell^p}$ for short. 

Since $\nabla_{\pmb{\theta}}\cU(\mathbf{0},\bX^\ep(t))=\mathbf 0$ as one can check, we have
\[
\nabla_{\pmb{\theta}}\cU(\pmb{\theta}^\ep(t),\bX^\ep(t))=\int_0^1\frac{d}{d\zeta}\nabla_{\pmb{\theta}}\cU(\zeta\pmb{\theta}^\ep(t),\bX^\ep(t))d\zeta =\int_0^1\text{Hess}_{\pmb{\theta}}\big[\cU(\zeta\pmb{\theta}^\ep(t),\bX^\ep(t))\big]d\zeta \cdot \pmb{\theta}^\ep(t).
\]
Furthermore, as discussed above the eigenvalues of $\text{Hess}_{\pmb{\theta}}$ are uniformly bounded by below by $\wt c$ when $\cG^{\bX^\ep(t)}$ is connected (and nonnegative in general), hence for any $\xi\in\R^N\setminus\{0\}$ the min-max theorem gives
\[
\Big\langle\xi,\int_0^1\text{Hess}_{\pmb{\theta}}\big[\cU(\zeta\pmb{\theta}^\ep(t),\bX^\ep(t))\big]d\zeta\cdot\xi\Big\rangle \geq \wt c 1\{\cG^{\bX^\ep(t)} \text{is connected}\}\|\xi\|_2^2.
\]
Putting together these two equations with \eqref{eq.dttheta} leads to
\begin{align*}
\frac{1}{2}\frac{d}{dt}\|\pmb{\theta}^\ep(t)\|_2^2&= -\frac{K}{N}\Big\langle\pmb{\theta}^\ep(t), \int_0^1\text{Hess}_{\pmb{\theta}}\big[\cU(\zeta\pmb{\theta}^\ep(t),\bX^\ep(t))\big]d\zeta \cdot \pmb{\theta}^\ep(t)\Big\rangle\\
&\le -\frac{\wt cK}{N}\; 1\{\cG^{\bX^\ep(t)} \text{is connected}\}\|\pmb{\theta}^\ep(t)\|_2^2.
\end{align*}
Above we have also used that $\Delta(\gamma)$ is convex, hence $\zeta\pmb{\theta}^\ep(t)$ lies in $\Delta(\gamma)$ whenever $\pmb{\theta}^\ep(t)$ does; 
and that if $\pmb{\theta}^\ep(0)\in\Delta(\gamma)$, then $\pmb{\theta}^\ep(t)\in\Delta(\gamma)$ for all $t\ge0$, established at the beginning. This yields our claim since
\[
\|\pmb{\theta}^\ep(t)\|_2^2\le \|\pmb{\theta}^\ep(0)\|_2^2 \exp\big\{-\frac{2\wt cK}{N} \int_0^t1\{\cG^{\bX^\ep(s)} \text{is connected}\}ds\big\}
\]
and $\int_0^\infty 1\{\cG^{\bX^\ep(t)} \text{is connected}\} dt =\infty$, $\P$-a.s. {This completes the proof for \abbr{KRW}.

Turning to \abbr{DRC}, note that besides simple notational changes, the only essential ingredient is that the set $\cH$ of all {\it{connected}} subgraphs of $\G$ is non-empty (in fact  $\G\in\cH$) and positive recurrent for the process $G^{\ep,\kappa}(t)=(V,\cE^{\ep,\kappa}(t))$, for any $\ep,\kappa\in(0,\infty)$. Since $\{\cE^{\ep,\kappa}(t)\}_{t\ge0}$ is a continuous-time irreducible Markov chain on a finite state space $\{0,1\}^{|E|}$, $\cH$ is positive recurrent. The jump rate of $\{\cE^{\ep,\kappa}(t)\}$ is at most $(\ep\vee\kappa)|E|$, thus positive recurrence is equivalent to having
\[
\int_0^\infty 1\{G^{\ep,\kappa}(t)\in\cH\}dt=\infty, \quad \P\text{-a.s.},
\]
and the proof follows.
}\end{proof}

\medskip

Next, we characterize the set of equilibria for both models: 

\begin{proof}[Proof of Proposition \ref{ppn:char-equil}]
(a) Consider \abbr{KRW}, by the positive recurrence of  $\{\bX^\ep(t)\}$, $\P$-a.s. there exists a collection of disjoint time-intervals $\{\Delta T^\ep_\bx\}_{\bx\in V^N}$ with $\Delta T^\ep_\bx\subset[0,\infty)$, such that
\begin{align*}
\bX^\ep(t)|_{t\in\Delta T^\ep_\bx}=\bx,
\end{align*}
whereby, for every $i$,
\begin{align*}
b_i(\cdot, \bX^\ep(t))|_{t\in\Delta T^\ep_\bx}=b_i(\cdot,\bx).
\end{align*}
Hence, if $\wh{{\pmb\theta}^\ep}$ is a fixed point of \eqref{dyn-km}, it must satisfy $b_i(\wh{{\pmb\theta}^\ep},\bx)=0$ for all $\bx\in V^N$ and $i$, $\P$-a.s. That is, $\wh{{\pmb\theta}^\ep}\in\bS_1$, $\P$-a.s.

Turning to \abbr{DRC}, by the positive recurrence of  $\{\cE^{\ep,\kappa}(t)\}$, $\P$-a.s. there exists a collection of disjoint time-intervals $\{\Delta T^{\ep,\kappa}_\cE\}_{\cE\subset E}$ with $\Delta T^{\ep,\kappa}_\cE\subset[0,\infty)$, such that
\begin{align*}
\cE^{\ep,\kappa}(t)|_{t\in\Delta T^{\ep,\kappa}_\cE}=\cE,
\end{align*}
whereby, for every $u\in V$,
\begin{align*}
b_u(\cdot, \cE^{\ep,\kappa}(t))|_{t\in\Delta T^{\ep,\kappa}_\cE}=b_u(\cdot,\cE).
\end{align*}
Hence, if $\wh{{\pmb\theta}^{\ep,\kappa}}$ is a fixed point of \eqref{kuramoto.RCM}, it must satisfy $b_u(\wh{{\pmb\theta}^{\ep,\kappa}},\cE)=0$ for all $\cE\subset E$ and $u\in V$, $\P$-a.s. That is, $\wh{{\pmb\theta}^{\ep,\kappa}}\in\bS_2$, $\P$-a.s.

(b) Restricting to the homogeneous setting, consider \abbr{KRW} and assume that there exists two vertices $u^*\neq v^*\in V$ such that $\pi(u^*,v^*)=0$. Recall also that    each vertex has a self-loop. Consider
\[
\bx_1:=(u^*,u^*,v^*,v^*,...,v^*)\in V^N
\]
i.e. all but its first two components are $v^*$, the first two being $u^*$. Then, if $\wh{{\pmb\theta}^\ep}\in\bS_1$, it must satisfy 
\begin{align*}
0=b_1(\wh{{\pmb\theta}^\ep},\bx_1)=\pi(u^*,u^*)\sin(\wh{{\theta}^\ep_2}-\wh{{\theta}^\ep_1}),
\end{align*}
implying that $|\wh{{\theta}^\ep_2}-\wh{{\theta}^\ep_1}|\in\{0,\pi\}$. Permuting the components of $\bx_1$, we can get $|\wh{{\theta}^\ep_j}-\wh{{\theta}^\ep_i}|\in\{0,\pi\}$ for every $i,j$. That is, $\wh{{\pmb\theta}^\ep}\in\bS$. The other direction $\bS\subset \bS_1$ is simple.

Turning to \abbr{DRC}, for any pair of vertices $u\neq v\in V$ such that $\pi(u,v)> 0$ in the skeleton, consider a subgraph of the skeleton graph whose edge set is given by
\[
\cE_{u,v}=\{(u,v)\}
\]
i.e. it consists of a single edge $(u,v)$. If $\wh{{\pmb\theta}^{\ep,\kappa}}\in\bS_2$, then it must satisfy
\begin{align*}
0=b_u(\wh{{\pmb\theta}^{\ep,\kappa}},\cE_{u,v})=\pi(u,v)\sin(\wh{{\theta}^{\ep,\kappa}_v}-\wh{{\theta}^{\ep,\kappa}_u}),
\end{align*}
implying that $|\wh{{\theta}^{\ep,\kappa}_v}-\wh{{\theta}^{\ep,\kappa}_u}|\in\{0,\pi\}$. Since the above $u,v$ with $\pi(u,v)>0$ is arbitrary, and the skeleton is connected, we get $\wh{{\pmb\theta}^{\ep,\kappa}}\in\bS$. The other direction $\bS\subset \bS_2$ is simple.
\end{proof}

Under the condition of part (b) of Proposition \ref{ppn:char-equil}, the set of possible fixed points of \eqref{dyn-km} or \eqref{kuramoto.RCM} is explicitly given by (up to a global translation)
\begin{align*}
\bS=\{\pmb\theta:\; \theta_i\in\{0,\pi\}, \; \forall i\}.
\end{align*}
Since the set $\bS$ is deterministic and has finite cardinality, it is of interest to investigate the stability of each of the elements of $\bS$. We adopt the following special notions of stability and instability for our dynamic problem. In particular, the notion of instability is very weak. See Conjecture \ref{conj:attractor} in this regard.
\begin{definition}
For the model \eqref{dyn-km} with $\pmb\omega=\mathbf0$ (and analogously for \eqref{kuramoto.RCM} with simple notation changes), we say that

(a) a fixed point $\pmb\theta^\star\in\bS$ is asymptotically stable if there exists some $\delta>0$ such that for any $t_0\ge0$, $\P$-a.s. we have that 
\begin{align*}
\|\pmb\theta^\ep(t_0)-\pmb\theta^\star\|<\delta\quad \Rightarrow \quad \lim_{t\to\infty}\pmb\theta^\ep(t)=\pmb\theta^\star,
\end{align*}
where $\|\pmb\theta\|:=\max_i|\theta_i|$.

(b) a fixed point $\pmb\theta^\star\in\bS$ is weakly unstable if there exists some $\delta>0$ such that for any $\iota>0$, $\P$-a.s. there exists $t_0\ge0$ and an initial condition satisfying $\|\pmb\theta^\ep(t_0)\|<\iota$ so that 
\begin{align*}
\sup_{t\ge t_0}\|\pmb\theta^\ep(t)-\pmb\theta^\star\|>\delta.
\end{align*}
\end{definition}

\begin{proposition}\label{prop:stab}
For \eqref{dyn-km} and \eqref{kuramoto.RCM} with $\pmb\omega=\mathbf0$, all elements of $\bS$ except for $\mathbf 0$ are weakly unstable, while $\mathbf 0$ is asymptotically stable.
\end{proposition}

\begin{proof}
We write the proof for \abbr{KRW} \eqref{dyn-km}, and the proof for \abbr{DRC} \eqref{kuramoto.RCM} is similar.

(a) For any $t_0\ge0$, Theorem \ref{thm:basin} applied to the time-shifted processes $\{\bX^\ep(\cdot+t_0),\pmb\theta^\ep(\cdot+t_0)\}$ implies that $\mathbf 0$ is asymptotically stable. Indeed, one can choose any $\delta <\pi/4$.

(b) Fix any $\pmb\theta^\star=(\theta^\star_1,...,\theta^\star_N)\in\bS\backslash\{\mathbf 0\}$ and a vector $\bx=(x_1,...,x_N)\in V^N$ such that the subgraph $\cG^\bx$ induced by $\bx$ has the following property: there exist $i,j$ such that $\pi(x_i,x_j)>0$ and $\theta^\star_i=0$, $\theta^\star_j=\pi$. Now we claim that $\pmb\theta^\star$ is an unstable fixed point for the static Kuramoto model with coefficients given by $a(\bx)$. Indeed, the Hessian matrix at $\pmb\theta^\star$ of the energy function \eqref{hessian} is not nonnegative definite. Consider $\bv=(v_1,v_2,...,v_N)$ given by  
\begin{align}\label{unstable-dir}
 v_i = 1\{\theta_i^\star=0\} - 1\{\theta_i^\star = \pi \}, \quad i=1,...,N.
\end{align}
A direct computation (already present in \cite[Prop. 4.1]{canale2008almost}) shows that
\begin{align*}
\mathbf v^T(\partial^2 \cU(\pmb \theta^\star,\bx) )\mathbf v &= -2\sum_{i,j=1}^N a_{ij}(\bx) 1\{\theta_i^\star \ne \theta_j^\star \}\\
&= -2\sum_{i,j=1}^N \pi(x_i,x_j) 1\{\theta_i^\star \ne \theta_j^\star \}<0.
\end{align*}
That is, $\bv$ is an unstable direction for the static Kuramoto model with coefficients $a(\bx)$. 

Fix $\delta=1$. For such static Kuramoto model, given any $\iota>0$, if we choose initial condition to be $\pmb\theta^\star+\iota \bv/2$, then due to instability the solution will exit a neighborhood of radius $\delta$ in finite time. We denote by $\tau(\iota,\bv)<\infty$ such an exit time. 

Now consider our dynamic problem. By the positive recurrence of $\{\bX^\ep(t)\}$, $\P$-a.s. there exists an infinite sequence of disjoint time-intervals 
\begin{align*}
\Lambda^\ep_{k,\bx}:=[t^\ep_{k,\bx}, t^\ep_{k,\bx}+\zeta^\ep_k)\subset[0,\infty),\quad k\in\N
\end{align*}
with $\{\zeta^\ep_k\}_{k\in\N}$ independent exponential random variables of mean $N^{-1}\ep$, such that 
\begin{align*}
\bX^\ep(t)|_{t\in \Lambda^\ep_{k,\bx}}=\bx.
\end{align*}
Given any $\iota>0$, by properties of i.i.d. exponential random variables, $\P$-a.s. there exists some $k_0\in\N$ such that $\zeta^\ep_{k_0}>\tau(\iota,\bv)$. Thus, if we choose $t_0:=t^\ep_{k_0,\bx}$ and $\pmb\theta^\ep(t_0)=\pmb\theta^\star+\iota \bv/2$, then we have both $\|\pmb\theta^\ep(t_0)\|<\iota$ and $\sup_{t\in[t_0,t_0+\tau(\iota,\bv)]}\|\pmb\theta^\ep-\pmb\theta^\star\|>\delta$. That is, we have shown that $\pmb\theta^\star$ is weakly unstable.
\end{proof}

Despite the rather weak conclusion regarding the instability of each of $\pmb\theta^\star\in\bS\backslash\{\mathbf 0\}$ that we obtained, we believe that much stronger is true. The difficulty is that although we have demonstrated in the above proof a non-random and time-independent direction $\bv$ that is unstable for $\pmb\theta^\star$ and all static Kuramoto problems (subject to a certain connectivity assumption), we are not able to prove that for our dynamic problem, the trajectory of the solution will not come back once it has left. Furthermore, we cannot handle the situation when the aforementioned connectivity assumption fails, and thus the direction $\bv$ is not necessarily unstable. The following conjecture states that for our dynamic problem, $\mathbf0$ attracts the whole phase space except for $\bS\backslash\{\mathbf0\}$, and there is even no stable manifolds around the latter set of spurious fixed points.

\begin{conjecture}\label{conj:attractor}
For the model \eqref{dyn-km} with $\pmb\omega=\mathbf 0$ (and analogously for \eqref{kuramoto.RCM} with notation changes), for any $\ep\in(0,\infty)$ and any deterministic choice of $\pmb\theta^\ep(0)\in\T^N\backslash\bS$, $\P$-a.s.  we have that (up to a global translation)
\begin{align*}
\lim_{t\to\infty}\pmb\theta^\ep(t)=\mathbf 0.
\end{align*}
\end{conjecture}

\section{An averaging principle for the \abbr{KRW}}\label{sec.aver}

In this section, we prove an averaging principle, as $\ep\to0$, for \abbr{KRW} \eqref{dyn-km}, i.e. Proposition \ref{prop:ave}. Recall the averaged equation \eqref{ave-rate}. First, we make some remarks and state a lemma and a proposition that will be used in the proof.

\begin{remark}\label{rmk:ave-km}
If $\pmb\omega=\mathbf0$, \eqref{ave-km} is a homogeneous Kuramoto model on a complete graph of $N$ vertices $\{1,...,N\}$ with equal positive weights $K\ovl{a}N^{-1}$, and it is well-known (cf. \cite[Corollary 17.6]{bullo2020lectures}, \cite{taylor2012there}, \cite[Section 3]{ling2019landscape}) that the basin of attraction $\Omega$ of the point $\mathbf{0}$ is the whole phase space $\T^N$ minus a set $\cN$ of Lebesgue measure zero. Hence the only stable equilibrium for the dynamics is $\mathbf{0}$ (up to a global translation). The set $\cN$ is the union of the stable manifolds of all the unstable equilibria, explicitly described in the proof of \cite[Corollary 17.6]{bullo2020lectures}.
\end{remark}

\begin{remark}
We observe that equation \eqref{ave-km} is similar to the one obtained in \cite[Eq. $(4.2)$]{medvedev2018continuum} but the averaging principle we obtain is different from the one obtained by Medvedev. While Medvedev's averaging principle is over the nodes of the graph our averaging is in time. In Medvedev's context the number of nodes grows to infinity while in ours the parameter going to infinity is the speed of the random walkers, keeping the number of nodes in the skeleton graph fixed.
\end{remark}

\begin{lemma}\label{rmk:lip}

(a) We have that $\ovl{a}\le A$ and $|\ovl{b}_i(\pmb{\theta})|\le KA+\|\pmb{\omega}\|_\infty$, for any $\pmb{\theta}, i$. \\

\noindent
(b) The function $\pmb{\theta}\mapsto \ovl{b}_i(\pmb{\theta})$ is Lipschitz and $\pmb{\theta}\mapsto b_i(\pmb{\theta},\bx)$ is uniformly (in $\bx$) Lipschitz. 
\noindent
\end{lemma}

\begin{proof}
(a) The first statement holds since $\mu(v)\le 1$ and
\[
\sum_{u,v\in V}\pi(u,v)\mu(u)\mu(v)\le \sum_{u,v\in V}\pi(u,v)\mu(u)\le \sum_{u\in V}\mu(u)\pi(u)\le A.
\]
The second one is immediate.

(b) Indeed, for any $j$ we have that
\begin{align*}
|\partial_{\theta_j}\ovl{b}_i(\pmb{\theta})|=\frac{K}{N}\Big|\ovl{a}\cos\big(\theta_j-\theta_i\big)1_{\{j\neq i\}}-\sum_{k\neq i}\ovl{a}\cos\big(\theta_k-\theta_i\big)1_{\{j=i\}}\Big|\le KA,
\end{align*}
implying that for some finite constant $C=C(K,N,A)$, $|\ovl{b}_i(\pmb{\theta})-\ovl{b}_i(\pmb{\theta'})|\le C\|\pmb{\theta}-\pmb{\theta'}\|_2$. Similar argument shows that $\pmb{\theta}\mapsto b_i(\pmb{\theta},\bx)$ is also uniformly Lipschitz.\\

\end{proof}

\begin{proposition} For the Kuramoto Random Walk $\{\bX^\ep(t)\}_{t\ge 0}$ we have, 
\[
\lim_{\ep\to0}\sup_{\pmb{\theta}\in\T^N}\E\left|\delta^{-1}\int_{0}^{\delta}\big[b_i(\pmb{\theta}, \bX^\ep(t))dt-\ovl{b}_i(\pmb{\theta})\big]dt\right|=0.
\]
\end{proposition}

\begin{proof}
 The process $\{\bX^\ep(t)\}$ has a unique invariant probability measure, namely the product measure $\mu^{\otimes N}$. Observe that the dependence of $\{\bX^\ep(t)\}$ in $\ep$ is only through the jump rates. In fact $\{\bX^\ep(t)\} = \{\bX^1(t\ep^{-1})\}$. Hence, the invariant measure is independent of $\ep$ and in a $\delta$-unit of time, for any fixed $\delta>0$, the number of jumps diverges with $\ep^{-1}$ a.s. By Ergodic theorem of continuous time finite Markov chains, for every $i,j$ and $\delta>0$, we have that
 \begin{align*}
 \lim_{\ep\to0}\;\delta^{-1}\int_{0}^{\delta}a_{ij}(\bX^\ep(t))dt&=\lim_{\ep\to0}\;\delta^{-1}\int_{0}^{\delta}\sum_{u,v\in V}\pi(u,v)1_{\{X^\ep_i(t)=u,X^\ep_j(t)=v\}}dt\\
&= \lim_{(\delta/\ep)\to \infty}\;\frac{1}{\delta/\ep} \int_{0}^{\delta/\ep}\sum_{u,v\in V}\pi(u,v)1_{\{X^1_i(t)=u,X^1_j(t)=v\}}dt\\
&=\sum_{u,v\in V}\pi(u,v)\mu(u)\mu(v)=\ovl{a}, \quad \P\text{-a.s.}
\end{align*}
which implies convergence in $L^1(\P)$ by the dominated convergence theorem. Since the state space $V^N$ is finite, the rate of convergence in $L^1$ can be taken uniform in the starting state. Indeed, for any starting distribution $\bX^\ep(0)\sim\nu$, we have
\begin{align*}
\E^\nu\big|\int_{0}^{\delta}\big[a_{ij}(\bX^\ep(t))-\ovl{a}\big]dt\big|\le \max_{\bx\in V^N}\E^\bx\big|\int_{0}^{\delta}\big[a_{ij}(\bX^\ep(t))-\ovl{a}\big]dt\big|,
\end{align*}
where $\E^\bx$ indicates conditioning on $\bX^\ep(0)=\bx$.
It also implies that for any $i$, $\pmb{\theta}\in\T^N$ and $\delta>0$, we have
\begin{align}
&\limsup_{\ep\to0}\E\big|\delta^{-1}\int_{0}^{\delta}\big[b_i(\pmb{\theta}, \bX^\ep(t))dt-\ovl{b}_i(\pmb{\theta})\big]dt\big| \label{conv-b}\\
&= \limsup_{\ep\to0}\E\big|\sum_{j=1}^N\delta^{-1}\int_{0}^{\delta}\big[a_{ij}(\bX^\ep(t))-\ovl{a}\big]dt\sin\big(\theta_j-\theta_i\big)\big|\nonumber\\
&\le \lim_{\ep\to0}\sum_{j=1}^N\E\big|\delta^{-1}\int_{0}^{\delta}\big[a_{ij}(\bX^\ep(t))-\ovl{a}\big]dt\big|=0.\nonumber
\end{align}
Further, note that
\begin{align*}
\pmb{\theta}\mapsto \E\big|\sum_{j=1}^N\delta^{-1}\int_{0}^{\delta}\big[a_{ij}(\bX^\ep(t))-\ovl{a}\big]dt\sin\big(\theta_j-\theta_i\big)\big|
\end{align*}
is uniformly continuous (with Lipschitz constant depending only on $N,A$), and $\T^N$ is compact, hence the convergence in \eqref{conv-b} can be taken to be uniform in $\pmb{\theta}$, i.e.
\begin{align}\label{unif-conv-b}
&\lim_{\ep\to0}\sup_{\pmb{\theta}\in\T^N}\E\big|\delta^{-1}\int_{0}^{\delta}\big[b_i(\pmb{\theta}, \bX^\ep(t))dt-\ovl{b}_i(\pmb{\theta})\big]dt\big|=0.
\end{align}
\end{proof}

We can now give the 
\begin{proof}[Proof of Proposition \ref{prop:ave}]
The proof strategy is well established in the averaging literature (cf. \cite[Chapter 2, Sections 1]{freidlin1998random}), we give the details for completeness. Assume $\pmb{\theta}^{\ep}(0)=\ovl{\pmb{\theta}}(0)=\pmb{\alpha}\in\T^N$. We divide $[0,T]$ into subintervals of length $\delta>0$ of the form $[k\delta,(k+1)\delta)$ for integers $k=0,1,...,[T/\delta]$. On each subinterval $t\in[k\delta,(k+1)\delta)$ and for any $i$ we have that
\begin{align*}
\theta^\ep_i(t)&=\theta^\ep_i(k\delta)+\omega_i(t-k\delta)+\int_{k\delta}^tb_i(\pmb{\theta}^\ep(s),\bX^\ep(s))ds\\
\ovl{\theta}_i(t)&=\ovl{\theta}_i(k\delta)+\omega_i(t-k\delta)+\int_{k\delta}^t\ovl{b}_i(\ovl{\pmb{\theta}}(s))ds
\end{align*}
hence, for any $t\in[k\delta,(k+1)\delta)$, taking the difference of the above, we get
\begin{align*}
\theta^\ep_i(t)-\ovl{\theta}_i(t) &=  [\theta^\ep_i(k\delta)-\ovl{\theta}_i(k\delta)]   +\int_{k\delta}^t[b_i(\pmb{\theta}^\ep(k\delta),\bX^\ep(s))-\ovl{b}_i(\pmb{\theta}^\ep(k\delta))]ds \\
								  & +\int_{k\delta}^t[\ovl{b}_i(\pmb{\theta}^\ep(s))-\ovl{b}_i(\ovl{\pmb{\theta}}(s))]ds\\
								  &+\int_{k\delta}^t[b_i(\pmb{\theta}^\ep(s),\bX^\ep(s))-b_i(\pmb{\theta}^\ep(k\delta),\bX^\ep(s))]ds \\
								  & -\int_{k\delta}^t[\ovl{b}_i(\pmb{\theta}^\ep(s))-\ovl{b}_i(\pmb{\theta}^\ep(k\delta))]ds.
\end{align*}
Recall Lemma \ref{rmk:lip} that $b_i, \ovl{b}_i$ are uniformly bounded and  uniformly Lipschitz in $\pmb{\theta}$, hence the third integral above can be bounded in absolute value as follows:
\begin{align*}
&\int_{k\delta}^t|b_i(\pmb{\theta}^\ep(s),\bX^\ep(s))-b_i(\pmb{\theta}^\ep(k\delta),\bX^\ep(s))|ds\le C(K,N,A)\int_{k\delta}^t|\pmb{\theta}^\ep(s)-\pmb{\theta}^\ep(k\delta)|ds\\
&\le C(K,N,A,\pmb{\omega})\int_{k\delta}^t(s-k\delta)ds\le C(t-k\delta)^2,
\end{align*}
and a similar computation holds for the last integral. Taking $t=(k+1)\delta$, using $\ovl{b}_i$ Lipschitz, we have that
\begin{align*}
|\theta^\ep_i((k+1)\delta)-\ovl{\theta}_i((k+1)\delta)|
&\le |\theta^\ep_i(k\delta)-\ovl{\theta}_i(k\delta)|+
\big|\int_{k\delta}^{(k+1)\delta}\big[b_i(\pmb{\theta}^\ep(k\delta),\bX^\ep(s))-\ovl{b}_i(\pmb{\theta}^\ep(k\delta))\big]ds\big|\\
&+C\int_{k\delta}^{(k+1)\delta}\|\pmb{\theta}^\ep(s)-\ovl{\pmb{\theta}}(s)\|_2ds+C\delta^2.
\end{align*}
For any $t\in[0,T]$, we sum the previous over $k=0, 1,...,[t/\delta]-1$, for some $C=C(K,N,A,\pmb{\omega})$ we have that 
\begin{align*}
|\theta^\ep_i(t)-\ovl{\theta}_i(t)|&\le \sum_{k=0}^{[t/\delta]}\big|\int_{k\delta}^{(k+1)\delta}\big[b_i(\pmb{\theta}^\ep(k\delta),\bX^\ep(s))-\ovl{b}_i(\pmb{\theta}^\ep(k\delta))\big]ds\big|+Ct\delta+C\int_{0}^t\|\pmb{\theta}^\ep(s)-\ovl{\pmb{\theta}}(s)\|_2ds,
\end{align*}
where the initial conditions cancel. Let us denote
\begin{align*}
\Xi(\ep,\delta,t):= \sum_{k=0}^{[t/\delta]}\big|\delta^{-1}\int_{k\delta}^{(k+1)\delta}\big[b_i(\pmb{\theta}^\ep(k\delta),\bX^\ep(s))-\ovl{b}_i(\pmb{\theta}^\ep(k\delta))\big]ds\big|,
\end{align*}
which is nondecreasing in $t$. It follows that
\[
\sup_{t\in[0,T]}\|\pmb{\theta}^\ep(t)-\ovl{\pmb{\theta}}(t)\|_2\le 
\Xi(\ep,\delta,T)\delta+CT\delta+C\int_{0}^T\sup_{s\in[0,t]}\|\pmb{\theta}^\ep(s)-\ovl{\pmb{\theta}}(s)\|_2dt.
\]
Taking expectation and by Gronwall's lemma, we have that 
\begin{align}\label{gronwall}
\E\big[\sup_{t\in[0,T]}\|\pmb{\theta}^\ep(t)-\ovl{\pmb{\theta}}(t)\|_2\big]\le \big(\E\big[\Xi(\ep,\delta,T)\big]+CT\big)\delta e^{CT}.
\end{align}
Upon conditioning at each $\cF_{k\delta}$ of the canonical filtration $\cF_t:=\sigma\{\bX^\ep(s), 0\le s\le t\}$, with $\pmb{\theta}^\ep(k\delta)$ measurable with respect to $\cF_{k\delta}$, by Markov property and \eqref{unif-conv-b} we have that
\begin{align*}
&\E\Big[\big|\delta^{-1}\int_{k\delta}^{(k+1)\delta}\big[b_i(\pmb{\theta}^\ep(k\delta),\bX^\ep(s))-\ovl{b}_i(\pmb{\theta}^\ep(k\delta))\big]ds\big|\; \Big|\; \cF_{k\delta}\Big]\\
&\le \max_{\bx\in V^N}\sup_{\pmb{\theta}\in\T^N}\E^\bx\big|\delta^{-1}\int_0^\delta\big[b_i(\pmb{\theta}, \wt\bX^\ep(s))-\ovl{b}_i(\pmb{\theta})\big]ds\big|\quad \to 0,\quad \text{ as }\ep\to0,
\end{align*}
where $\E^\bx$ indicates conditioning on $\wt\bX^\ep(0)=\bx$, for $\{\wt\bX^\ep(t)\}$ process having the same law as $\bX^\ep(t)$ (to avoid confusion). We see from the \abbr{RHS} that the rate of convergence can be taken universal, independent of $\cF_{k\delta}$ and hence of the initial condition $\pmb{\theta}^\ep(0)=\pmb{\alpha}$. 

Applying the tower property and dominated convergence theorem, we further have 
\[
\E\big|\delta^{-1}\int_{k\delta}^{(k+1)\delta}\big[b_i(\pmb{\theta}^\ep(k\delta),\bX^\ep(s))-\ovl{b}_i(\pmb{\theta}^\ep(k\delta))\big]ds\big|\to0, \quad \text{ as }\ep\to0
\]
with a universal rate of convergence, and upon summing over $k=0, 1,...,[T/\delta]-1$, we get
\begin{align*}
\lim_{\ep\to0}\sup_{\pmb{\alpha}\in\T^N}\E\big[\Xi(\ep,\delta,T)\big]=0.
\end{align*}
Thus, we deduce from \eqref{gronwall} that
\begin{align*}
\limsup_{\ep\to0}\sup_{\pmb{\alpha}\in\T^N}\E\big[\sup_{t\in[0,T]}\|\pmb{\theta}^\ep(t)-\ovl{\pmb{\theta}}(t)\|_2\big]\le CT\delta e^{CT}.
\end{align*}
Since the \abbr{LHS} does not depend on $\delta$, taking $\delta\to0$ we get our thesis for $\ell^2$-norm. Since $\ell^\infty$-norm is bounded above by $\ell^2$-norm in finite dimension $N$, the thesis is proved.
\end{proof}

\section{Results that hold for small $\ep>0$}\label{sec.global}

In this section, we use the averaging principle in Section \ref{sec.aver} and the result on synchronization of Section \ref{sec.anyep} to prove Theorem \ref{thm-main}:


\begin{proof}[Proof of Theorem \ref{thm-main}]
By Remark \ref{rmk:ave-km}, the averaged dynamics \eqref{ave-km} has a basin of attraction $\Omega$ for $\mathbf{0}$ (up to a global translation) of full Lebesgue measure on $\T^N$. Consider any compact subset $\Omega'\subset\joinrel\subset\Omega$. By compactness of $\Omega'$, the continuity with respect to initial condition of the deterministic \abbr{ODE} \eqref{ave-km}, and the fact that $\ovl{\pmb{\theta}^{\pmb{\alpha}}}(t) \to \mathbf{0}$ for every $\pmb{\alpha} \in \Omega '$, there exists some finite, deterministic $t_0=t_0(\Omega', \ovl{a}, N)$ such that for any initial condition $\ovl{\pmb{\theta}^{\pmb{\alpha}}}(0)=\pmb{\alpha}\in\Omega'$ and all $t\ge t_0$,
\[
\ovl{\pmb{\theta}^{\pmb{\alpha}}}(t)\in \Delta (\pi/6).
\]
Take $T=2t_0$. By Proposition \ref{prop:ave} and Markov's inequality, for any $\eta>0$, there exists $\ep_0=\ep_0(t_0,\eta)\in(0,1)$ such that for any $\ep\in(0,\ep_0)$, we have
\begin{align*}
\sup_{\pmb{\alpha}\in\Omega'}\P\Big(\sup_{t\in[0,T]}\|\pmb{\theta}^{\ep,\pmb{\alpha}}(t)-\ovl{\pmb{\theta}^{\pmb{\alpha}}}(t)\|_\infty> \frac{\pi}{12}\Big)<\eta,
\end{align*}
where $\pmb{\theta}^{\ep,\pmb{\alpha}}(0)=\pmb{\alpha}$.

It follows by the triangle inequality that for any $\ep\in(0,\ep_0)$ and $\pmb{\alpha}\in\Omega'$, with probability at least $1-\eta$ we have 
\[
\pmb{\theta}^{\ep,\pmb{\alpha}}(t)\in\Delta(\pi/3)
\]
for all $t\in(t_0,T]$. By Theorem \ref{thm:basin}, $\Delta(\pi/3)$ is a basin of attraction for $\mathbf{0}$ (up to a global translation) for the \abbr{KRM} \eqref{dyn-km}, $\P$-a.s. for any $\ep$. Hence, $\pmb{\theta}^{\ep,\pmb{\alpha}}(t) \to \mathbf{0}$ as $t\to\infty$, with probability at least $1-\eta$.
\end{proof}

\medskip
\textbf{Acknowledgments.} We thank Franco Flandoli and Gabriel Mindlin for valuable discussions on random dynamical systems and synchronization respectively. We also thank Steven Strogatz, who pointed us to the problem treated in this manuscript and to the references \cite{ling2019landscape, Faggian2019, Multi_Review}.

Pablo Groisman and Hern\'an Vivas are partially supported by CONICET grant PIP 2021 11220200102825CO. 

Pablo Groisman is partially supported by UBACyT grant 20020190100293BA.

R.H. was supported in part by a CRM-Pisa Junior Visiting Position 2020-22, and by the Deutsche Forschungsgemeinschaft (DFG, German Research Foundation) under Germany's Excellence Strategy EXC 2044-390685587, Mathematics M\"unster: Dynamics-Geometry-Structure.



\end{document}